\documentclass[12pt, a4]{amsart}
\usepackage{geometry}
\usepackage{amsfonts, amstext, amscd, amsmath, mathrsfs, amssymb, amsthm, amsxtra, bbding, epsfig, graphicx, latexsym, url, color}
\usepackage{mathtools}
\usepackage[all]{xy}

\usepackage{tikz-cd}

\usepackage{tikz}
\usepackage[normalem]{ulem}
\usepackage{soul}
\usepackage{color}

\setstcolor{red}

\setstcolor{blue}

\setstcolor{green}
\makeatletter
\@namedef{subjclassname@2020}{%
	\textup{2020} Mathematics Subject Classification}
\makeatother
\usepackage{hyperref}
\hypersetup{colorlinks=true,linkcolor=cyan,anchorcolor=cyan,citecolor=cyan}

\def \C {{\mathbb C}}

\def \N {{\mathbb N}}

\def \R {{\mathbb R}}

\def \C {{\mathbb C}}

\def \i {{\rm i}}

\def \J {{\mathcal J}}

\def \N {{\mathbb N}}
\def \M {{\mathcal M}}

\def \R {{\mathbb R}}

\def \lcm {{\rm lcm}}
\def \gcd {{\rm gcd}}

\def \d {\,{\rm d}}
\def\re{{\Re e\,}}
\def\im{{\Im m\,}}

\def\geq{\geqslant}
\def\le{\leqslant}
\def\ge{\geqslant}

\allowdisplaybreaks

\theoremstyle{plain}
\newtheorem{theorem}{Theorem}[section]
\newtheorem{proposition}{Proposition}[section]
\newtheorem{lemma}[proposition]{Lemma}
\newtheorem{corollary}[theorem]{Corollary}

\theoremstyle{remark}

\numberwithin{equation}{section}
\addtocounter{footnote}{1}

\numberwithin{equation}{section}

\addtocounter{footnote}{1}

\begin{document}
	
	\title[]
	{On large values of $|\zeta(\sigma+{\rm i}t)|$}
	\author[Zikang Dong]{Zikang Dong}
	\address{%
		CNRS LAMA 8050\\
		Laboratoire d'analyse et de math\'ematiques appliqu\'ees\\
		Universit\'e Paris-Est Cr\'eteil\\
		61 avenue du G\'en\'eral de Gaulle\\
		94010 Cr\'eteil Cedex\\
		France
	}
	\email{zikangdong@gmail.com}
	\author[Bin Wei]{Bin Wei}
	\address{Center for Applied Mathematics, Tianjin University, Tianjin 300072, P.R. China}
	\email{bwei@tju.edu.cn}

	\date{\today}
	
	
	\begin {abstract}
	In this article, we investigate the extreme values of the Riemann zeta function $\zeta(s)$.
	On the 1-line, we obtain a lower bound evaluation
	$$\max_{t\in[T^{\beta},T]}|\zeta(1+\i t)|\ge {\rm e}^\gamma(\log_2T+\log_3T+c),$$
	with an effective constant $c$ which improves the result of Aistleitner, Mahatab and Munsch.
	When $\sigma\searrow 1/2$, we get an improved $c(\sigma)$ in the evaluation
	$$\max_{t\in[T^{\beta},T]}\log|\zeta(\sigma+\i t)|\ge c(\sigma)\frac{(\log T)^{1-\sigma}}{(\log_2T)^\sigma},$$
	which improves the result of Bondarenko and Seip.
    This is based on an improved lower bound of G\'al-type sums.
\end{abstract}

\maketitle

\section{Introduction}

In this article, we investigate the extreme values of the Riemann zeta function $\zeta(s)$ in the strip $1/2<\re s\le1$. The study of the values on the 1-line can date back to 1925 when Littlewood \cite{Li25} showed that there exists arbitrarily large $t$ for which
\begin{align*}
	|\zeta(1+{\rm i}t)|\geq\{1+o(1)\}{\rm e}^\gamma\log_2t .
\end{align*}
Here and throughout, we denote by $\log_j$ the $j$-th iterated logarithm and by $\gamma$ the Euler constant.
This was improved by Levinson \cite{Le72}, who in 1972 proved that there exists arbitrarily large $t$ such that
$$ |\zeta(1+{\rm i}t)|\geq{\rm e}^\gamma\log_2t +O(1).$$
In 2006, Granville and Soundararajan \cite{GS06} used Diophantine approximation to prove that the lower bound
$$\max_{t\in[1,T]}|\zeta(1+{\rm i}t)|\geq{\rm e}^\gamma(\log_2T+\log_3T-\log_4T+O(1))$$
holds for sufficiently large $T$.
Based on a distribution function which exhibits slightly smaller values, they also proposed a strong conjecture that
\begin{align}\label{conj}
	\max_{t\in[T,2T]}|\zeta(1+{\rm i}t)|={\rm e}^\gamma(\log_2T+\log_3T+C_0+1-\log2+o(1)),
\end{align}
for some constant $C_0=-0.3953997$, which provides a precise description of the extreme values.
For the upper bound, the best unconditional result is established by Vinogradov \cite{Vi58} who proved that
$$|\zeta(1+\i t)|\ll(\log t)^{2/3}.$$

In recent years, the resonance method has been extensively developed, which can detect extreme values of the Riemann zeta function more effectively. It was first used by Voronin \cite{Vo88} in 1988, and developed by Soundararajan \cite{S08} in 2008 and Hilberdink \cite{Hil09} in 2009 separately. In 2018, Aistleitner, Mahatab and Munsch \cite {Ai19} used a variant ``long resonance'' to show that
\begin{equation}\label{result:AMM}
\max_{t\in[\sqrt T,T]}|\zeta(1+{\rm i}t)|\geq{\rm e}^\gamma(\log_2T+\log_3T+O(1)).
\end{equation}
Note that this requires a lager range $[\sqrt T,T]$ {than $[T, 2T]$ in \eqref{conj},
which} is typical in the application of ``long resonance''.
Their method can also apply to a class of generalized $L$-functions (see \cite{dix}).
Inspired by their work, the aim of this article is to get an improved lower bound of large values, which presents an explicit description of the error term $O(1)$ .
For any $\beta\in(0,1)$, we define
$$Z_\beta(\sigma,T):=\max_{ T^\beta\le|t|\le T}|\zeta(\sigma+\i t)|.$$
Then {we} have the following theorem.

\begin{theorem}\label{th1.4}
Let $0<\beta<1$ be fixed and $c$ be a constant such that $c<\log(1-\beta)-\log_24-1$.
Then we have
$$
Z_\beta(1,T)\ge{\rm e}^\gamma(\log_2T+\log_3T+c)
$$
for sufficiently large $T$, where $\gamma$ is the Euler constant.
\end{theorem}

When $\beta=1/2$, we can choose the constant $c=-2.0197814$.
This gives a description of the error term $O(1)$ in the result {\eqref{result:AMM}} by Aistleitner, Mahatab and Munsch. When $\beta\to 0$, we can further choose $c=-1.32663426$.
Despite the enlarged range, this is comparable with the conjecture \eqref{conj} which predicts a larger constant
$C_0+1-\log2=-0.0885469$. Theorem \ref{th1.4} is also in accordance with the results on the Dirichlet $L$-functions, due to Aistleitner, Mahatab, Munsch and Peyrot \cite{Ai2019}.

Now we turn our attention to the values of the Riemann zeta function in the strip $1/2<\re s<1$. For any fixed $\sigma\in(1/2,1)$, Titchmarsh \cite{Ti28} in 1928 showed that for any $\varepsilon>0$ there exist arbitrarily large $t$ such that
$$\log |\zeta(\sigma+{\rm i}t)|\ge(\log t)^{1-\sigma-\varepsilon}.$$
In 1972, Levinson \cite{Le72} improved this result by showing that for large $T$ there exists a positive $c$ such that
$$\max_{t\in[0,T]}\log |\zeta(\sigma+{\rm i}t)|\ge c\frac{(\log T)^{1-\sigma}}{\log_2T}.$$
In 1977, Montgomery \cite{Mon77} showed that
\begin{align}
	\max_{t\in[0,T]}\log |\zeta(\sigma+{\rm i}t)|\ge c(\sigma)\frac{(\log T)^{1-\sigma}}{(\log_2T)^\sigma},\label{10}
\end{align}
where $c(\sigma)=\sqrt{\sigma-1/2}/20$ unconditionally, and $c(\sigma)=1/20$ under the Riemann hypothesis.
In 2016, using the resonance method, Aistleitner \cite{A16} improved Montgomery's unconditional result by showing that (\ref{10}) holds for $c(\sigma)=0.18(2\sigma-1)^{1-\sigma}$.

For the upper bound, Richert \cite{Ri1967} in 1967 proved that
$$
|\zeta(\sigma+\i t)|\le At^{B(1-\sigma)^{3/2}}(\log t)^{2/3}
$$
holds for some absolute $A$ and $B$.
Successive improvements of Richert's bound have reduced the admissible size of $A$ and $B$.
We refer to Cheng \cite{Cheng1999} and Ford \cite{Ford2002} for more details. Under the Riemann Hypothesis, we could have a much better upper bound (see \cite{La2011} and \cite{Ti86})
\begin{align}\label{upperbound}
\log |\zeta(\sigma+\i t)|\ll\frac{(\log t)^{2-2\sigma}}{\log_2t}.
\end{align}

It is conjectured that the true order of the magnitude of $\max_{t\in[0,T]}\log |\zeta(\sigma+{\rm i}t)|$ corresponds to the lower bound (\ref{10}) rather than the upper bound (\ref{upperbound}).
In 2011, based on the probabilistic model, Lamzouri \cite{La2011} gave an explicit conjectural value of $c(\sigma)$, claiming that
$$
\max_{t\in[0,T]}\log |\zeta(\sigma+{\rm i}t)|= c(\sigma)\frac{(\log T)^{1-\sigma}}{(\log_2T)^\sigma}
$$
holds with
\begin{align}
c(\sigma)
:= \frac{1}{\sigma^{2\sigma}(1-\sigma)^{1-\sigma}}\int_0^\infty{\frac{\log I_0(t)}{t^{{1}/{\sigma}+1}}} {\rm d}t,\label{LAM}
\end{align}
where $I_0(t)$ is the modified Bessel function of order $0$.

In 2018, Bondarenko and Seip \cite{BS18} made celebrated improvement on this topic.
They proved that there exists a function {$\nu(\sigma)$}
which is bounded below by $1/(2-2\sigma)$ and has the asymptotic behavior
\begin{align}
{\nu(\sigma)} = \left\{
\begin{array}{rcl}
\big\{1/\sqrt2+o(1)\big\} \sqrt{|\log(2\sigma-1)|} && {\sigma\searrow1/2},
\\\noalign{\vskip 2mm}
(1-\sigma)^{-1}+O(|\log(1-\sigma)|) && {\sigma\nearrow1}.
\end{array}\right.
\label{csigma}
\end{align}
Here $\sigma\searrow1/2$ means $\sigma$ tends to 1/2 from above {with $1/2+1/\log_2T\le\sigma\le 3/4$}
and $\sigma\nearrow1$ means $\sigma$ tends to 1 from {below} {with $3/4\le\sigma\le1-1/\log_2T$},
as $T\to+\infty$.
Then for $1/2+1/\log_2{T}\le\sigma\le3/4$,
\begin{align}\label{BK1}
\max_{t\in[\sqrt{T},T]}\log |\zeta(\sigma+{\rm i}t)|\ge {\nu(\sigma)}\frac{(\log T)^{1-\sigma}}{(\log_2T)^\sigma},
\end{align}
and for $3/4\le\sigma\le1-1/\log_2{T}$,
\begin{align}\label{BK2}
\max_{t\in[T/2,T]}\log |\zeta(\sigma+{\rm i}t)|\ge
{\log_3T + c} + {\nu(\sigma)}\frac{(\log T)^{1-\sigma}}{(\log_2T)^\sigma},
\end{align}
{where $c$ is an absolute constant.}

In this article, we aim to improve their first result (\ref{BK1}).
This is accomplished by deriving a lower bound for the maximum of G\'al-type sums, which is a kind of certain greatest common divisor (GCD) sums of the form
\begin{align}\label{GCDsumdef}
S_\sigma(\M)
:= \sum_{m,n\in \M} \frac{\gcd(m, n)^\sigma}{\lcm(m,n)^\sigma}
= \sum_{m,n\in \M} \frac{\gcd(m, n)^{2\sigma}}{(mn)^\sigma}.
\end{align}
For any positive integer $N$, we denote
\begin{align}
	\Gamma_\sigma(N):=\sup_{|\M|=N}\frac{S_\sigma(\M)}{N}.\label{211}
\end{align}
{A brief historic description on G\'al-type sums will be presented in the subsection 4.1.
By adapting the argument of de la Bret\`eche and Tenenbaum \cite{DT18} about $\Gamma_{1/2}(N)$,
we have the following theorem concerning the lower bound of $\Gamma_\sigma(N)$.}

\begin{theorem}\label{th241}
As $\sigma\searrow1/2$, we have
$$
\Gamma_\sigma(N)
\ge \exp\bigg(\big\{2\sqrt2+o(1)\big\}\sqrt{|\log(\sigma-1/2)|}\frac{(\log N)^{1-\sigma}}{(\log_2N)^\sigma}\bigg).
$$
\end{theorem}

The following theorem sets up the relation between large values of the Riemann zeta function and the G\'al-type sums.
\begin{theorem}\label{th3}
For {$\beta\in(0,1)$ and} ${\sigma\in[1/2+1/\log_2T,1)}$, we have
$$
Z_\beta(\sigma,T)\gg\sqrt{\Gamma_\sigma(T^{1-\beta})},
$$
{where the implied constant is absolute.}
\end{theorem}

As a direct deduction of Theorem \ref{th241}, we have the following corollary.

\begin{corollary}\label{}
Let $0<\beta<1$ be fixed and let $\frac{1}{2}<\sigma<1$. Then we have that
$$
\max_{t\in[T^\beta,T]}\log |\zeta(\sigma+{\rm i}t)|
\ge {c_{\beta}(\sigma)}\frac{(\log T)^{1-\sigma}}{(\log_2T)^\sigma}
$$
holds for a function {$c_{\beta}(\sigma)$} which has the asymptotic behavior
$$
{c_{\beta}(\sigma)} = \big\{\sqrt2+o(1)\big\}(1-\beta)^{1-\sigma}\sqrt{|\log(\sigma-1/2)|},$$
as $\sigma\searrow1/2$.
\end{corollary}

Therefore when $\beta=1/2$ and $\sigma\searrow1/2$, we improved the result of Bondarenko and Seip by a factor $2^\sigma$.
It is also worthy noting {that} Lamzouri's conjecture \eqref{LAM} predicts $c(\sigma)\sim (2\sigma-1)^{-1/2}$ as $\sigma\searrow1/2$.

For the sake of completeness, we also mention the large values of the Riemann zeta function on the critical line.
For the lower bound, de la Br\`eteche and Tenenbaum \cite{DT18} have shown that
$$
\max_{t\in[0,T]}|\zeta(1/2+\i t)|
\ge \exp\left(\big\{\sqrt2+o(1)\big\}\sqrt{\frac{\log T\log_3T}{\log_2T}}\right),
$$
improving earlier results made by Bondarenko and Seip \cite{BS17,BS19}.
For the upper bound, the {Lindel\"of} hypothesis states that for any $\varepsilon>0$
$$\zeta(1/2+\i t)\ll_{\varepsilon} t^\varepsilon,$$
while the best-known upper bound is due to Bourgain \cite{Bou17} who proved that
$$\zeta(1/2+\i t)\ll_{\varepsilon} t^{13/84+\varepsilon}.$$
However, we have the conjectural value due to Famer, Gonek and Hugh \cite{FGH07} which asserts that
$$
\max_{t\in[0,T]}|\zeta(1/2+\i t)|
= \exp\big(\big\{\sqrt2/2+o(1)\big\}\sqrt{\log T\log_2T}\big).
$$
For much earlier work, we refer to \cite{Ba86,Rama77,CS11,S08,Ti86}.

This article is organized as follows. In \S\ref{sec2}, we introduce some preliminary lemmas. We establish the approximation of the Riemann zeta function by its truncated Euler product. In \S\ref{sec3}, we discuss the large values of the Riemann zeta function $\zeta(s)$ on 1-line and establish Theorem \ref{th1.4}. In \S\ref{sec4}, we make a brief review on the G\'al-type sums and then prove Theorem \ref{th241}. Finally in \S\ref{sec5}, we connect G\'al-type sums to the values of the Riemann zeta function, and establish Theorem \ref{th3}.

\bigskip
\section{Preliminary lemmas}\label{sec2}

In this section, we introduce some preliminary lemmas.
We start with Mertens' formula with an explicit error term.
\begin{lemma}\label{l121} {\it
		Let $x>1000$, then we have
		$$\frac1{{\rm e}^\gamma\log x}\left(1-\frac{1}{2(\log x)^2}\right)\le\prod_{p\le x}\left(1-\frac{1}{p}\right) \le\frac1{{\rm e}^\gamma\log x}\left(1+\frac{1}{2(\log x)^2}\right).$$

	}
\end{lemma}
\begin{proof}
	See Theorem 7 of \cite{RS62}.
\end{proof}

The following lemma plays a key role in the proof of Theorem \ref{th241}.
\begin{lemma}\label{l122} {\it
		Let $x$ be large, then we have
		$$\sum_{p\le x}\frac{1}{p^\sigma}= \{1+o(1)\}\frac{x^{1-\sigma}}{(1-\sigma)\log x}.$$

	}
\end{lemma}
\begin{proof}
	The is essentially Lemma 6 of \cite{BS18}. The only adjustment lies in the use of prime number theory, where we replace the lower bound by the asymptotic formula
$$\pi(x)=\{1+o(1)\}\frac{x}{\log x}.$$
For analogous statements, see also Lemma 3.1 of \cite{Nor92}, equation (2.1) of \cite{La2011}, and Lemma 3.3 of \cite{BG13}.
\end{proof}

The following lemma can be seen as a generalization of the greatest common divisors to rational numbers.

\begin{lemma}\label{l241}
Let $a,a', b,b'\in \N^*$ such that $\gcd(a,b)=1$ and $\gcd(a',b')=1$.
Then for any $N\in\N^*$ satisfying $b\mid N$ and $b'\mid N$, we have
$$
\frac{1}{N}\gcd\bigg(\frac{Na}{b},\frac{Na'}{b'}\bigg)=\frac{\gcd(a,a')}{\lcm(b, b')}
\qquad\text{and}\qquad
\frac{1}{N}\lcm\bigg(\frac{Na}{b},\frac{Na'}{b'}\bigg)=\frac{\lcm(a,a')}{\gcd(b,b')}.
$$
	As a direct inference, we have
	\begin{align*}
		\frac{\gcd(Na/b,Na'/b')}{\lcm(Na/b,Na'/b')}=\frac{\gcd(a,a')}{\lcm(a,a')}\frac{\gcd(b,b')}{\lcm(b, b')}.
	\end{align*}
\end{lemma}

\begin{proof}
	Since $\lcm(b,b')\mid N$, we may write
\begin{align*}
\gcd\left(\frac{Na}{b},\frac{Na'}{b'}\right)
& = N\frac{\gcd(a,a')}{\lcm(b, b')}
\gcd\bigg(\frac{\lcm(b, b')}{b}\frac{a}{\gcd(a,a')},\frac{\lcm(b, b')}{b'}\frac{a'}{\gcd(a,a')}\bigg)
\\
& = N\frac{\gcd(a,a')}{\lcm(b, b')}
\gcd\bigg(\frac{b'}{\gcd(b,b')}\frac{a}{\gcd(a,a')},\frac{b}{\gcd(b,b')}\frac{a'}{\gcd(a,a')}\bigg).
\end{align*}
Then the first assertion follows by the assumptions of co-primeness and the second one follows by the simple fact that
$$
\gcd\bigg(\frac{Na}{b},\frac{Na'}{b'}\bigg)\cdot{\rm lcm}\bigg(\frac{Na}{b},\frac{Na'}{b'}\bigg)
= N^2 \frac{aa'}{bb'}
= N^2\frac{\gcd(a,a')\lcm(a,a')}{\gcd(b,b')\lcm(a,a')}.
$$
\end{proof}

For convenience, we denote the truncated Euler product of the Riemann zeta function by
	$$
	\zeta(s; y) := \prod_{p\le y} (1-p^{-s})^{-1}.$$
The following lemma approximates $\zeta(s)$ on the 1-line by its truncated Euler product.

\begin{lemma}\label{l61}
Let $0<\beta<1$ be fixed. Then we have
$$
\zeta(1+{\rm i}t) = \zeta(1+{\rm i}t;Y) \bigg\{1+O\bigg(\frac{1}{(\log T)^{1/\beta}}\bigg)\bigg\}
$$
uniformly for $T\ge 3$, $T^{\beta}\le |t|\le T$ and
$Y=\exp\{(\log T)^{1/\beta}\}$,
where the implied constant depends on $\beta$ only.
\end{lemma}

\begin{proof}
{
Let $s=1+{\rm i}t$ with $T^{\beta}\le |t|\le T$. Firstly, when $\re w>1$, we have
$$
\log\zeta(w)
= \log\prod_p\bigg(1-\frac{1}{p^w}\bigg)^{-1}
= \sum_p\sum_{\nu\ge1}\frac{1}{\nu p^{\nu w}}
= \sum_{n\ge 2} \frac{\Lambda(n)}{n^w\log n},
$$
where $\Lambda(n)$ is the von Mangoldt function.
Set $\sigma_1=1/\log Y$, $\sigma_0=-c/\log T$ for suitably positive constant $c=c(\beta)>0$ and
$T_0=T^{\beta}/2$.
Denote the contour joining
$\sigma_1-{\rm i}T_0$,
$\sigma_1+{\rm i}T_0$,
$\sigma_0+{\rm i}T_0$,
 $\sigma_0-{\rm i}T_0$ and
 $\sigma_0+{\rm i}T_0$ by $\Gamma$, i.e.,
\begin{align*}
\oint_\Gamma
= \int_{\sigma_1-{\rm i}T_0}^{\sigma_1+{\rm i}T_0}
		+\int_{\sigma_1+{\rm i}T_0}^{\sigma_0+{\rm i}T_0}
		+\int_{\sigma_0+{\rm i}T_0}^{\sigma_0-{\rm i}T_0}
		+\int_{\sigma_0-{\rm i}T_0}^{\sigma_1-{\rm i}T_0}.
	\end{align*}
	Since $\log\zeta(1+{\rm i}t+w)Y^w$ is analytic inside $\Gamma$, by Cauchy's integral formula we have
\begin{equation}\label{int/Gamma}
\frac{1}{2\pi {\rm i}}\oint_\Gamma\log\zeta(1+{\rm i}t+w)\frac{Y^w}{w}\d w=\log\zeta(1+{\rm i}t).
\end{equation}
	Then by Perron's formula (\cite[Corollary II.2.1]{Te95} with $s=1+{\rm i}t$, $\sigma_a=\alpha=1$,
$\kappa=1/\log Y$ and $B(x)=1$), we have
$$
\frac{1}{2\pi {\rm i}}\int_{\sigma_1-{\rm i}T_0}^{\sigma_1+{\rm i}T_0} \log\zeta(1+{\rm i}t+w)\frac{Y^w}{w}\d w
= \sum_{2\le n\le Y} \frac{\Lambda(n)}{n^{1+{\rm i}t}\log n}+O\bigg(\frac{\log T}{T_0}\bigg).
$$
On the other hand, noticing that
\begin{align*}
\sum_{p\le Y, \, p^{\nu}>Y} \frac{1}{\nu p^{\nu}}
\le \frac{1}{\log Y} \sum_{p\le Y, \, \nu\ge 2} \frac{\log p}{p^{\nu}}
\ll \frac{1}{\log Y},
\end{align*}
we can write
\begin{align*}
\sum_{2\le n\le Y} \frac{\Lambda(n)}{n^{1+{\rm i}t}\log n}
& = \sum_{p^{\nu}\le Y} \frac{1}{\nu p^{\nu(1+{\rm i}t)}}
= \sum_{p\le Y} \sum_{\nu\ge 1} \frac{1}{\nu p^{\nu(1+{\rm i}t)}}
+ O\bigg(\frac{1}{\log Y}\bigg)
\\
& = \log \zeta(1+{\rm i}t; Y)
+ O\bigg(\frac{1}{\log Y}\bigg).
\end{align*}
Inserting this into the preceding formula, we get
\begin{equation}\label{61}
\frac{1}{2\pi {\rm i}}\int_{\sigma_1-{\rm i}T_0}^{\sigma_1+{\rm i}T_0} \log\zeta(1+{\rm i}t+w)\frac{Y^w}{w}\d w
= \log \zeta(1+{\rm i}t; Y)+O\bigg(\frac{1}{(\log T)^{1/\beta}}\bigg).
\end{equation}

For the other three integrals,
in view of bounds of $\log \zeta(w)$ in the zero-free region (see \cite[Theorem II.3.16]{Te95}), typically we have
$$
\bigg(\int_{\sigma_1+{\rm i}T_0}^{\sigma_0+{\rm i}T_0}
		+\int_{\sigma_0+{\rm i}T_0}^{\sigma_0-{\rm i}T_0}\bigg)\log\zeta(1+{\rm i}t+w)\frac{Y^w}{w}\d w
\ll \frac{\log T}{T_0},
$$
	and
$$
\int_{\sigma_0+{\rm i}T_0}^{\sigma_0-{\rm i}T_0} \log\zeta(1+{\rm i}t+w)\frac{Y^w}{w}\d w
\ll\frac{(\log T)^2}{\exp(c(\log T)^{(1/\beta-1)})}.
$$
	Thus the lemma follows from \eqref{int/Gamma}, \eqref{61} and these two bounds.}
\end{proof}

\bigskip

\section{Extreme values of $|\zeta(1+{\rm i}t)|$: Proof of Theorem 1.1}\label{sec3}

Choose $B$ such that ${\rm e}^{c+1}<B<(1-\beta)/\log4$ and set $X:=B\log T\log_2T$.
{Denote by $\mathcal{S}(X)$ the set of all $X$-friable numbers.}
Let $a_n=a(n)$ be the completely multiplicative function supported on $\mathcal{S}(X)$ with $a_p=1-p/X$ for $p\le X$
and $a_p=0$ otherwise.
Define the resonator
$$
R(t) = \sum_{n\in \mathcal{S}( X)} a_nn^{{\rm i}t}
= \prod_{p\le X} (1-a_p{p^{{\rm i}t}})^{-1}.
$$
	Then by the prime number theorem, {for $t\in \R$} we have
$$
\log|R({t})|
\le \sum_{p\le X}\log(1-a_p)^{-1}
= \pi(X)\log X-\theta(X),
$$
where $\theta(x)$ is the Chebyshev function. It is well known that
$$
\pi(x)\log x-\theta(x)=\{1+o(1)\}\frac{x}{\log x}
$$
and thus we have
\begin{equation}\label{UB:R(t)}
|R({t})|\le T^{B+o(1)}
\qquad
{(t\in \R)}.
\end{equation}

    Set the weight function to be the Gaussian function $\phi(t)={\rm e}^{-t^2}$ which satisfies
	$$\widehat\phi(x)=\int_{\mathbb R}\phi(t){\rm e}^{-{\rm i}tx}\d t=\sqrt{2\pi}\phi(x).$$
	Denote
	$$M_1(R; T) := \int_{T^\beta\le |t|\le T}|R(t)|^2\phi\bigg(\frac{t\log T}{T}\bigg)\d t,$$
	and
	$$M_2(R; T) := \int_{T^\beta\le |t|\le T}\zeta(1+{\rm i}t)|R(t)|^2\phi\bigg(\frac{t\log T}{T}\bigg)\d t.$$
	Then clearly we have
	\begin{align}
		Z_\beta(T)\ge\frac{|M_2(R; T)|}
		{M_1(R; T)}.\label{64}		
	\end{align}
Put $Y:=\exp((\log T)^{1/\beta})$ as in Lemma \ref{l61}. It follows that
\begin{align*}
M_2(R; T)
= \bigg\{1+O\bigg(\frac{1}{(\log T)^{1/\beta}}\bigg)\bigg\}
\int_{T^\beta\le |t|\le T}\zeta(1+{\rm i}t;Y)|R(t)|^2\phi\bigg(\frac{t\log T}{T}\bigg)\d t.
\end{align*}
Using \eqref{UB:R(t)} and  the trivial bound
$$
|\zeta(1+{\rm i}t;Y)|\ll\log Y=(\log T)^{1/\beta}
\qquad
{(t\in \R)},
$$
we can deduce that
\begin{align*}
	\left|\int_{|t|\le T^\beta}\zeta(1+{\rm i}t;Y)|R(t)|^2\phi\bigg(\frac{t\log T}{T}\bigg)\d t\right|
	\le T^{2B+\beta+o(1)},
\end{align*}
and
\begin{align*}
\bigg|\int_{|t|\ge  T}\zeta(1+{\rm i}t;Y)|R(t)|^2\phi\bigg(\frac{t\log T}{T}\bigg)\d t\bigg|
\ll T^{2B+o(1)} \int_T^{\infty} {\rm e}^{-(t\log T/T)^2} \d t
\ll1.
\end{align*}
Thus by denoting
$$
I_2(R; T) = \int_{\mathbb R}\zeta(1+{\rm i}t;Y)|R(t)|^2\phi\bigg(\frac{t\log T}{T}\bigg)\d t,$$
we have
$$
M_2(R; T)
= \bigg\{1+O\bigg(\frac{1}{(\log T)^{1/\beta}}\bigg)\bigg\}
\big\{I_2(R; T)+O\big(T^{2B+\beta-1+o(1)}\big)\big\}.
$$
On the other hand, we have
$$
M_1(R; T)
\le \int_{\mathbb R}|R(t)|^2\phi\bigg(\frac{t\log T}{T}\bigg)\d t
=: I_1(R; T),
$$
Therefore by (\ref{64}), we derive that
\begin{equation}\label{66}
Z_\beta(T)
\ge \bigg\{1+O\bigg(\frac{1}{(\log T)^{1/\beta}}\bigg)\bigg\}
\frac{I_2(R; T)+O(T^{2B+\beta-1+o(1)})}{I_1(R; T)}.		
\end{equation}

We first estimate the ratio $I_2(R; T)/I_1(R; T)$.
For $I_1(R; T)$, we have
\begin{align}
I_1(R; T)
& = \int_{\mathbb R} \sum_{m,n\in \mathcal{S}(X)}
a_ma_n\bigg(\frac m n\bigg)^{{\rm i}t} \phi\bigg(\frac {t\log T}{T}\bigg)\d t
\nonumber
\\
& = \frac{T}{\log T} \sum_{m,n\in \mathcal{S}(X)}a_ma_n {\widehat\phi}\bigg(\frac{T}{\log T}\log\Big(\frac{n}{m}\Big)\bigg)
\ge\widehat\phi(0)\frac{T}{\log T}\sum_{n\in \mathcal{S}(X)} a_n^2,
\label{67}
\end{align}
{thanks to the positivity of $a_n$ and $\widehat\phi(t)$.}
Similarly, for $I_2(R; T)$ we have
$$
I_2(R; T)
= \frac{T}{\log T}
\sum_{l\in \mathcal{S}(Y)} \sum_{m,n\in \mathcal{S}(X)} \frac{a_ma_n}{l} \widehat\phi\bigg(\frac{T}{\log T}\log\frac{nl}{m}\bigg).
$$
By sifting the terms with $l|m$ {and noticing $Y>X$}, we have that
\begin{align*}
I_2(R; T)
& \ge \frac{T}{\log T} \sum_{l,n,k\in \mathcal{S}(X)} \frac{a_{kl}a_n}{l} \widehat\phi\bigg(\frac{T}{\log T}\log\frac{n}{k}\bigg)
\\
& = \frac{T}{\log T} \sum_{l\in \mathcal{S}(X)} \frac{a_{l}}{l} \sum_{k,n\in \mathcal{S}(X)} a_ka_n
\widehat\phi\bigg(\frac{T}{\log T}\log\frac{n}{k}\bigg).
\end{align*}
Therefore, we can deduce that
$$
\frac{I_2(R; T)}{I_1(R; T)}
\ge \sum_{l\in \mathcal{S}(X)}\frac{a_{l}}{l}
= \prod_{p\le X}\bigg(1-\frac{a_p}{p}\bigg)^{-1}
= \prod_{p\le X} \bigg(1-\frac{1}{p}\bigg)^{-1} \prod_{p\le X}\bigg(\frac{p-1}{p-a_p}\bigg).
$$
For the first product, we use Lemma \ref{l121} to derive that
$$
\prod_{p\le X} \bigg(1-\frac{1}{p}\bigg)^{-1}
\ge {\rm e}^\gamma\log X \bigg(1-\frac{1}{2(\log X)^2}\bigg).
$$
For the second one, since
\begin{align*}
-\log\prod_{p\le X}\left(\frac{p-1}{p-a_p}\right)
\le\sum _{p\le X}\frac{p}{(p-1)X}
\le \bigg\{1+{O\bigg(\frac{1}{\log X}\bigg)}\bigg\}\frac{1}{\log X},	
\end{align*}
we have
$$
\prod_{p\le X} \bigg(\frac{p-1}{p-a_p}\bigg)
\ge 1-\frac{1}{\log X}+O\bigg(\frac{1}{(\log X)^2}\bigg).
$$
It follows that
\begin{equation}\label{61}
\frac{I_2(R; T)}{I_1(R; T)}\ge{\rm e}^\gamma \bigg\{\log X-1+{O\bigg(\frac{1}{\log X}\bigg)}\bigg\}.
\end{equation}

It remains to exclude the influence of the error term in \eqref{66}.
In view of \eqref{67}, we deduce that
\begin{align}
\log\sum_{n\in \mathcal{S}(X)}a_n^2
& = \log\prod_{p\le X} \big(1-a_p^2\big)^{-1}
= \log\prod_{p\le X} \big(1-(1-p/X)^2\big)^{-1}
\nonumber\\
& = 2\pi(X)\log X-\theta(X)-\sum_{p\le X}\log(2X-p).\label{62}
\end{align}
By partial summation formula, we can derive that
\begin{align}
	\sum_{p\le X}\log(2X-p)=\pi(X)\log X+\int_1^{X/(\log X)^2}\frac{\pi(t)}{2X-t}\d t+\int_{X/(\log X)^2}^X\frac{\pi(t)}{2X-t}\d t.\label{63}
\end{align}
For the first integral, we have
$$
\int_1^{X/(\log X)^2}\frac{\pi(t)}{2X-t}\d t
\ll \frac{X}{(\log X)^3}\int_1^{X/(\log X)^2}\frac{\d t}{2X-t}
\ll \frac{X}{(\log X)^{{5}}}.
$$
For the second one, the interval guarantees that
$$
\pi(t) = \frac{t}{\log X} \bigg\{1+O\bigg(\frac{\log_2X}{\log X}\bigg)\bigg\}.
$$
Thus we have
\begin{align*}
\int_{X/(\log X)^2}^X\frac{\pi(t)}{2X-t}\d t
= \{2\log2-1+o(1)\} \frac{X}{\log X}.
\end{align*}
Combining with (\ref{62}) and (\ref{63}), we obtain that
$$
\log\sum_{n\in \mathcal{S}(X)} a_n^2
= \{2-2\log2+o(1)\} \frac{X}{\log X},
$$
which implies by (\ref{67}) that
$$	I_1(R; T)\gg T^{(2-2\log2)B+1+o(1)}.$$
In view of (\ref{66}), for the error term we have
$$
\frac{T^{2B+\beta-1+o(1)}}{I_1(R; T)}
\le \frac{T^{2B+\beta+o(1)}}{T^{(2-2\log2)B+1+o(1)}}
= T^{B\log 4-(1-\beta)+o(1)},
$$
where the last exponent is negative and thus admissible.
Substituting this and \eqref{61} into (\ref{66}), we deduce that
\begin{align*}
	Z_\beta(T)
	\ge{\rm e}^\gamma(\log_2T+\log_3T+c),
\end{align*}
where $c{<}\log B-1$.
This completes the proof.
\hfill
$\square$

\bigskip

\section{Lower bound for $\Gamma_\sigma(N)$ as $\sigma\searrow1/2$: Proof of Theorem \ref{th241}}\label{sec4}

\subsection{A brief review on G\'al-type sums}\

\vskip 1mm
{In this subsection, we make a brief review on the G\'al-type sums.}
The study of $\Gamma_\sigma(N)$ arises naturally in metric Diophantine approximation.
When $\sigma=1$, this was a prize problem posed by the Dutch Mathematical Society in 1947 on Erd\"{o}s's suggestion.
G\'al \cite{Gal} investigated the problem in 1949 and proved that
$$\Gamma_1(N)\ll (\log_2N)^2.$$
Thereafter, the GCD sums (\ref{GCDsumdef}) are also known as ``G\'al-type sums''.
In 2017, Lewko and Radziw\l\l  { }\cite{Lew17} used the method of probabilitic models to give a much easier proof of G\'al's theorem.
They further determined the implied constant and proved that as $N\to\infty$, one has
$$
\Gamma_1(N)=\bigg\{\frac{{\rm e}^{2\gamma}}{\zeta(2)}+o(1)\bigg\}(\log_2N)^2.
$$

{Let $\M$ be a finite set of integers.}
For general $\sigma$, define the spectral norm of the GCD matrix
$({\gcd(m, n)^\sigma}/{\lcm(m,n)^\sigma})_{(m, n)\in \mathcal{M}^2}$ as
$$
Q_\sigma(\M)
:= \sup_{\substack{\boldsymbol{c}\in\C^{|\M|}\\ \|\boldsymbol{c}\|_2=1}}
\bigg|\sum_{m,n\in \M}c_m\overline{c_n}\frac{\gcd(m, n)^\sigma}{\lcm(m,n)^\sigma}\bigg|,
$$
where ${\boldsymbol{c}} := (c_1,\dots,c_N)\in \C^N$ and its norm ${\|}\boldsymbol{c}\|_2 := \sum_{j=1}^N|c_j|^2$.
Then {in} 2015, Aistleitner, Bondarenko and Seip \cite{ABS15} showed that
$$\Gamma_{1/2}(N)\le\sup_{|\M|=N}Q_{1/2}(\M)\le({\rm e}^2+1)(\log N+2)\max_{n\le N}\Gamma_{1/2}(n).$$
Recently, de la Bret\`eche and Tenenbaum \cite{DT18} gave asymptotic formulas for $\Gamma_{1/2}(N)$. Namely, they proved that
\begin{align}\label{gamma12}
\Gamma_{1/2}(N)
= \exp\left(\big\{2\sqrt2+o(1)\big\}\sqrt\frac{\log N\log_3N}{\log N}\right),
\end{align}
as $N\to\infty$.
For further details about GCD sums, we refer to \cite{ABS15,BHS2016,BS15,Lew17}.

{
Since Theorem \ref{th241} is a lower bound of $\Gamma_{\sigma}(N)$,
we only need to construct a set of integers $\M$ with $|\M|\le N$ such that
$$
\frac{S_{\sigma}(\M)}{|\M|}
\ge \exp\bigg(\big\{2\sqrt2+o(1)\big\}\sqrt{|\log(\sigma-1/2)|}\frac{(\log N)^{1-\sigma}}{(\log_2N)^\sigma}\bigg)
$$
for $1/2+1/\log_2N\le \sigma<1$ and $N\to\infty$.
Next we prove this by adapting the method of de la Bret\`eche and Tenenbaum \cite{DT18}.}

\subsection{Construction of the set $\M$}\

\vskip 1mm
Let $\alpha\in(1,+\infty)$, ${\eta}\in (0,+\infty)$, $f\in(1,{\rm e}]$ and $\lambda\in(0,1)$ {be some parameters}.
For $1\le j\le J:=\lfloor(\sigma-1/2)^{-\lambda}\rfloor$, define
$$
I_j := (f^j(\log N)\log_2N, f^{j+1}(\log N)\log_2N].
$$
Then for each interval $I_j$, we have that
\begin{equation}\label{size:Pj}
P_j := \sum_{p\in I_j}1
= (f-1)f^j(\log N)\bigg\{1+O_{\varepsilon}\bigg(\frac{j+\log_3N}{\log_2N}\bigg)\bigg\},
\end{equation}
{where we have supposed that $f\ge 1+((\log N)\log_2N)^{-5/12+\varepsilon}$
such that the prime number theorem in short intervals holds and the implied constant depends on $\varepsilon$ only.
On the other hand we note that the hypothesis $\sigma\ge 1/2+1/\log_2N$ {guaranties} $J\le (\log_2N)^{\lambda}=o(\log_2N)$.}

Let $N_j=\prod_{p\in I_j}p$ and $\omega(\cdot)$ {counts} the number of different prime factors. Set
$$
u_j := \bigg\lfloor\frac{{\eta}(\log N)^{1-\sigma}}{jf^{j(\sigma-1/2)}\sqrt {\log J}(\log_2N)^{\sigma}}\bigg\rfloor,
\qquad
v_j := \bigg\lfloor\frac{\alpha\log N}{j^2\log J}\bigg\rfloor.
$$
Then we define
$$
\M_j:=\bigg\{m=\frac{N_ja}{b} \,:\, ab\mid N_j \;\, \text{and} \;\, \omega(a),\,\omega(b)\le v_j\bigg\}
$$
and
\begin{align*}
	\M:=\prod_{1\le j\le J}\M_j=\bigg\{m=\prod_{1\le j\le J}m_j:\,m_j\in\M_j\,\,(1\le j\le J)\bigg\}.
\end{align*}

Now we evaluate the cardinal of $\M$. For $1\le j\le J$, we have
\begin{equation}\label{M_j}
|\M_j|=\sum_{\substack{0\le k\le v_j\\ 0\le \ell\le v_j}} \binom{P_j}{k}\binom{P_j-k}{\ell}.
\end{equation}
For fixed $k$, since $P_j$ is much larger than $v_j$, for $0\le \ell\le v_j$ we can deduce that
$$
{\binom{P_j-k}{v_j}
= \frac{(P_j-k-v_j)\cdots (P_j-k-(\ell+1))}{v_j\cdots (\ell+1)} \binom{P_j-k}{\ell}
\ge 2^{v_j-\ell} \binom{P_j-k}{\ell}.}
$$
Therefore, we have
\begin{equation}\label{inequality}
\sum_{0\le \ell\le v_j} \binom{P_j-k}{\ell}
\le {\sum_{0\le \ell\le v_j} 2^{-(v_j-\ell)} \binom{P_j-k}{v_j}}
\le 2 \binom{P_j-k}{v_j}.
\end{equation}
On the other hand, we have
$$
{\binom{P_j}{k}\binom{P_j-k}{v_j} = \binom{P_j}{v_j} \binom{P_j-v_j}{k}}.
$$
Thus, by \eqref{M_j} and \eqref{inequality}, we deduce that
\begin{align}\label{cardinalMj}
|\M_j|
\le 2\sum_{0\le k\le v_j} \binom{P_j}{k}\binom{P_j-k}{v_j}
{= 2 \binom{P_j}{v_j} \sum_{0\le k\le v_j} \binom{P_j-v_j}{k}}
\le 4 \binom{P_j}{v_j} \binom{P_j-v_j}{v_j}.
\end{align}
Using the Euler-Maclaurin formula, we have
\begin{align*}
	&|\M|\le\prod_{1\le j\le J}4\binom{P_j-v_j}{v_j}\binom{P_j}{v_j}\le\prod_{1\le j\le J}4\left(\frac{{\rm e}(P_j-v_j)}{v_j}\right)^{v_j}\left(\frac{{\rm e}P_j}{v_j}\right)^{v_j}\le&\prod_{1\le j\le J}4\left(\frac{{\rm e}P_j}{v_j}\right)^{2v_j}.
\end{align*}
Therefore, by \eqref{size:Pj} and the definitions of $v_j$ and $J$, we have
\begin{align*}
\log|\M|\le \sum_{1\le j\le J}2v_j \{j\log f+O({\log j+\log_4N})\}
\le 2\alpha\log f\log N+o(\log N),
\end{align*}
and consequently we obtain that
\begin{align}
	|\M|\le N^{2\alpha\log f+o(1)}.\label{cardinalM}
\end{align}

\subsection{Proof of Theorem \ref{th241}}\

\vskip 1mm
Note that $ab\mid N_j$ implies $\gcd(a,b)=1$ since $N_j$ is square-free. By Lemma \ref{l241}, we have
\begin{align}\label{245}
S_\sigma(\M_j)
= \sum_{\substack{a,a'\mid N_j\\\omega(a), \, \omega(a')\le v_j}} \frac{\gcd(a,a')^\sigma}{\lcm(a,a')^\sigma}\sum_{\substack{b,b'\mid N_j\\ \gcd(a, b)=\gcd(a', b',)=1\\  \omega(b), \, \omega(b')\le v_j}}
\frac{\gcd(b,b')^\sigma}{\lcm(b,b')^\sigma}.
\end{align}
Denote the inner sum by $\widetilde{S}_\sigma=\widetilde{S}_\sigma(a, a')$,
and let $\varphi_{2\sigma}(d)$ be the Euler's totient function of order $2\sigma$,
satisfying $\sum_{d\mid D}\varphi_{2\sigma}(d)=D^{2\sigma}$.
Then we have
\begin{align*}
\widetilde{S}_\sigma
& =\sum_{\substack{D,b,b'|N_j\\\gcd(b,a)=\gcd(b',a')=1\\\gcd(b,b')=D\\  \omega(b),\omega(b')\le v_j}}\frac{D^{2\sigma}}{(bb')^\sigma}
= \sum_{\substack{d\mid N_j\\\gcd(d,aa')=1\\\omega(d)\le v_j\\}}\varphi_{2\sigma}(d)\sum_{\substack{b,b'|N_j\\\gcd(b,a)=\gcd(b',a')=1\\d|\gcd(b,b')\\  \omega(b),\omega(b')\le v_j}}\frac{1}{(bb')^\sigma}
\\
&=\sum_{\substack{d\mid N_j\\\gcd(d,aa')=1\\\omega(d)\le v_j\\}}\frac{\varphi_{2\sigma}(d)}{d^{2\sigma}}\sum_{\substack{B,B'|N_j\\\gcd(B,ad)=\gcd(B',a'd)=1\\  \omega(B),\omega(B')\le v_j-\omega(d)}}\frac{1}{(BB')^\sigma}.
\end{align*}
By the definition of {$\varphi_{2\sigma}(d)$}, for $d\mid N_j$ we have
\begin{align}
	\frac{\varphi_{2\sigma}(d)}{d^{2\sigma}}=\prod_{p\mid d}\left(1-\frac{1}{p^{2\sigma}}\right)
	\ge \prod_{p\mid N_j}\left(1-\frac{1}{p^{2\sigma}}\right)\gg1.\label{247}
\end{align}
Consequently, we derive that
$$\widetilde{S}_\sigma\gg \sum_{\substack{d\mid N_j\\\gcd(d,aa')=1\\\omega(d)\le v_j\\}}\sum_{\substack{B,B'|N_j\\\gcd(B,ad)=\gcd(B',a'd)=1\\  \omega(B),\omega(B')\le v_j-\omega(d)}}\frac{1}{(BB')^\sigma}.$$
Substitute this into (\ref{245}). For the sum over $a$ and $a'$, we follow a similar procedure and derive that
\begin{align}
	S_\sigma(\M_j)&\gg\sum_{\substack{c\mid N_j\\\omega(c)\le v_j}}\sum_{\substack{A,A'|N_j\\\gcd(A,c)=\gcd(A',c)=1\\\omega(A),\omega(A')\le v_j-\omega(c)}}\frac{1}{(AA')^\sigma}\sum_{\substack{d\mid N_j\\\gcd(d,AA'c)=1\\\omega(d)\le v_j\\}}\sum_{\substack{B,B'|N_j\\\gcd(B,Acd)=\gcd(B',A'cd)=1\\  \omega(B),\omega(B')\le v_j-\omega(d)}}\frac{1}{(BB')^\sigma}\nonumber\\
	&=\sum_{\substack{c\mid N_j\\\omega(c)\le v_j}}\sum_{\substack{d\mid N_j\\\gcd(d,c)=1\\\omega(d)\le v_j\\}}\sum_{\substack{A|N_j\\\gcd(A,cd)=1\\\omega(A)\le v_j-\omega(c)}}\frac{1}{A^\sigma}\sum_{\substack{A'|N_j\\\gcd(A',cd)=1\\\omega(A')\le v_j-\omega(c)}}\frac{1}{A'^\sigma}\sum_{\substack{B|N_j\\\gcd(B,Acd)=1\\  \omega(B)\le v_j-\omega(d)}}\frac{1}{B^\sigma}\sum_{\substack{B'|N_j\\\gcd(B',A'cd)=1\\  \omega(B')\le v_j-\omega(d)}}\frac{1}{B'^\sigma}.\label{246}
\end{align}
We calculate from inside successively. Since each term is positive, we can sift a suitable subset.
Therefore we restrict $c$, $d$ such that $\omega(c)=\omega(d)= v_j-u_j$, and $A$, $A'$ such that $\omega(A)=\omega(A')=u_j$.
Then the inner sum turns to
\begin{equation}\label{2410}
\sum_{\substack{B'|N_j\\\gcd(B',A'cd)=1\\  \omega(B')\le u_j}}\frac{1}{B'^\sigma}
\ge \frac{1}{u_j!} \bigg(\sum_{\substack{p\in I_j\\ p\nmid A'cd}} \frac{1}{p^\sigma}\bigg)^{u_j}
= \frac{1}{u_j!} \bigg(\sum_{p\in I_j} \frac{1}{p^\sigma}-\sum_{\substack{p\in I_j\\ p|A'cd}} \frac{1}{p^\sigma}\bigg)^{u_j}.
\end{equation}
For the factorial, we use Stirling's formula
$$u_j!=\exp\left(u_j\log u_j-u_j+O(\log u_j)\right)=\left(\{1+o(1)\}\frac{u_j}{\rm e}\right)^{u_j}.$$
For the sums, by Lemma \ref{l122} we have
$$\sum_{p\in I_j}\frac{1}{p^\sigma}=\{1+o(1)\}\frac{(f^{1-\sigma}-1)f^{j(1-\sigma)}(\log N)^{1-\sigma}}{(1-\sigma)(\log_2N)^\sigma}.$$
Since $\omega(A'cd)\le\omega(A')+\omega(c)+\omega(d)\le2v_j$, we have
$$
\sum_{\substack{p\in I_j\\ p|A'cd}} \frac{1}{p^\sigma}
\le \sum_{ p|A'cd}\frac{1}{p^\sigma}\le\frac{2v_j}{(f^j\log N\log_2N)^\sigma}
\le \frac{2\alpha(\log N)^{1-\sigma}}{j^2f^{j\sigma}\log J(\log_2N)^\sigma}
\ll \frac{1}{j^2f^j\log J}\sum_{p\in I_j}\frac{1}{p^\sigma}.
$$
Note that $J\to\infty$ as $\sigma\searrow1/2$. Therefore, in (\ref{2410}) we have
\begin{align*}
	\sum_{\substack{B'|N_j\\\gcd(B',A'cd)=1\\  \omega(B')\le u_j}}\frac{1}{B'^\sigma}\ge\left(\{1+o(1)\}\frac{{\rm e}(f^{1-\sigma}-1)jf^{j/2}\sqrt{\log J}}{{\eta}(1-\sigma)}\right)^{u_j}.
\end{align*}
We can play similar trick on sums over $B, A', A$ in \eqref{246} successively and therefore
\begin{align}
	S_\sigma(\M_j)\gg	\left(\{1+o(1)\}\frac{{\rm e}(f^{1-\sigma}-1)jf^{j/2}\sqrt {\log J}}{{\eta}(1-\sigma)}\right)^{4u_j}\sum_{\substack{c\mid N_j\\\omega(c)= v_j-u_j}}
	\sum_{\substack{d\mid N_j\\\gcd(d,c)=1\\\omega(d)= v_j-u_j\\}}1.
	\label{SMj}
\end{align}
Trivially, we have
\begin{align*}
	\sum_{\substack{c\mid N_j\\\omega(c)= v_j-u_j}}\sum_{\substack{d\mid N_j\\\gcd(d,c)=1\\\omega(d)= v_j-u_j\\}}1
   \ge\binom{P_j}{v_j-u_j}\binom{P_j-v_j+u_j}{v_j-u_j}
	\ge \binom{P_j}{v_j}\binom{P_j-v_j}{v_j}\left(\frac{v_j}{P_j}\right)^{2u_j}.
\end{align*}
Therefore, by (\ref{cardinalMj}) we have
\begin{align*}
\sum_{\substack{c\mid N_j\\ \omega(c)\le v_j-u_j}} \sum_{\substack{d\mid N_j\\\gcd(d,c)=1\\\omega(d)\le v_j-u_j\\}}1	
\ge \frac{|\M_j|}{4} \bigg(\frac{v_j}{P_j}\bigg)^{2u_j}
= \frac{|\M_j|}{4} \bigg(\{1+o(1)\}\frac{\sqrt\alpha}{j(f-1)^{1/2}f^{j/2}\sqrt{\log J}}\bigg)^{4u_j}.
\end{align*}
Combined with (\ref{SMj}), we obtain that
$$
\frac{S_\sigma(\M_j)}{|\M_j|}
\gg \left(\{1+o(1)\}\frac{{\rm e}\sqrt\alpha(f^{1-\sigma}-1)}{{\eta}(1-\sigma)\sqrt{f-1}}\right)^{4u_j}.$$
By the definition of $u_j$, we have
\begin{align*}
	\sum_{j\le J}u_j&=\frac{4{\eta}(\log N)^{1-\sigma}}{\sqrt{\log J}(\log_2N)^\sigma}\sum_{j\le J}\frac{1}{jf^{j(\sigma-1/2)}}+O(J)\\
	&=\{1+o(1)\}\frac{4{\eta}\sqrt{\log J}(\log N)^{1-\sigma}}{f^{J(\sigma-1/2)}(\log_2N)^\sigma}\\
	&=\{1+o(1)\}\frac{4{\eta}\sqrt{\lambda|\log(\sigma-1/2)|}(\log N)^{1-\sigma}}{f^{(\sigma-1/2)^{1-\lambda}}(\log_2N)^\sigma}.
\end{align*}
Therefore, by taking product over $j$, we obtain
\begin{align*}
\frac{S_\sigma(\M)}{|\M|}
\ge \exp\bigg(\{1+o(1)\} H(\alpha, {\eta}, f, \lambda) \sqrt{|\log(\sigma-1/2)|} \frac{(\log N)^{1-\sigma}}{(\log_2N)^\sigma}\bigg),
\end{align*}
where
$$
H(\alpha, {\eta}, f, \lambda)
:= \frac{4{\eta}\sqrt{\lambda}}{f^{(\sigma-1/2)^{1-\lambda}}}
\log\bigg(\frac{{\rm e}\sqrt\alpha(f^{1-\sigma}-1)}{{\eta}(1-\sigma)\sqrt{f-1}}\bigg).
$$
Note that (\ref{cardinalM}) implies we need to restrict $2\alpha\log f\le1$. To get large value of $H$, we set
$$
f\to1^+,
\qquad
2\alpha\log f\to1^-,
\qquad
{\eta}\to\sqrt2/2,
\qquad
\lambda\to1^-.
$$
Then $H(\alpha, {\eta}, f, \lambda)$ can be sufficiently close to $2\sqrt2$ and we finally derive that
$$
\Gamma_\sigma(N)
\ge \frac{S_\sigma(\M)}{|\M|}
\ge \exp\bigg(\big\{2\sqrt2+o(1)\big\}\sqrt{|\log(\sigma-1/2)|}\frac{(\log N)^{1-\sigma}}{(\log_2N)^\sigma}\bigg).
$$

\bigskip
\section{Convolution method: Proof of Theorem \ref{th3}}\label{sec5}

Let $\M$  be a set of positive integers with cardinal $|\M|=N=T^\kappa$ where $0<\kappa<1$ to be chosen. Define
$$
\M_j
:= \M\cap [(1+(\log T)/T)^j, (1+(\log T)/T)^{j+1}).
$$
For $\J := \{j\ge0:\M_j\neq\varnothing\}$,
let $$\M'=\{m_j=\min\M_j:j\in\J\}.$$
Then we define the resonator
$$
R(t) := \sum_{m\in\M'} r(m)m^{\i t},
$$
where $r(m_j)=|\M_j|^{1/2}.$ Trivially we have
$$
|R(t)|
\le R(0)
= \sum_{m\in\M'} r(m)
\le \Big(\sum_{m\in\M'}1\Big)^{1/2} \Big(\sum_{m\in\M'}r(m)^2\Big)^{1/2}
\le |\M'|^{1/2}|\M|^{1/2}\le N.
$$

{Let $0<\varepsilon<1$}, for $u\in\R$, we take
$$K(u):=\frac{\sin^2(\varepsilon u\log T)}{\pi u^2\varepsilon\log T},$$
which satisfies
\begin{align}\label{Khat}
\widehat K(\xi) = \max\bigg(1-\frac{|\xi|}{2{\varepsilon}\log T}, 0\bigg).
\end{align}
Write
$${\mathfrak Z}_\sigma(t,u):=\zeta(\sigma+\i t+\i u)\zeta(\sigma-\i t+\i u)K(u).$$
Then we define
\begin{align*}
M_1(R; T)
& := \int_{T^\beta\le|t|\le T}|R(t)|^2{\phi}\bigg(\frac{t\log T}{T}\bigg)\d t,
\\
M_2(R; T)
& := \int_{2T^\beta\le|t|\le T/2}|R(t)|^2{\phi}\bigg(\frac{t\log T}{T}\bigg)\int_{|u|\le|t|/2}{\mathfrak Z}_\sigma(t,u)\d u\d t.
\end{align*}
Since $K(\cdot)$ is bounded by $1$, clearly we have
\begin{align}
	Z_\beta(\sigma,T)^2\ge\frac{{|M_2(R; T)|}}{M_1(R; T)}.\label{zbetasigma}
\end{align}
As in \S \ref{sec3}, we approximate $M_1(R; T)$ and $M_2(R; T)$ by their relative full integral, i.e., by
\begin{align*}
I_1(R; T)
& :=\int_{\R}|R(t)|^2{\phi}\bigg(\frac{t\log T}{T}\bigg)\d t,
\\
I_2(R; T)
& := \int_{\R}|R(t)|^2{\phi}\bigg(\frac{t\log T}{T}\bigg)\int_{\R}{\mathfrak Z}_\sigma(t,u)\d u\d t.
\end{align*}

\begin{lemma}\label{LI'}
For $I_1(R; T)$, we have
$$
M_1(R; T)\le I_1(R; T)\ll T|\M|/\log T,
$$
{where the implied constant is absolute.}
\end{lemma}

\begin{proof}
	The first inequality is trivial. Further we have
	\begin{align*}
		I_1(R; T)
		=\frac{T}{\log T}\sum_{i,j\in\J}r(m_i)r(m_j)\widehat\phi\bigg(\frac{T}{\log T}\log\frac{m_j}{m_i}\bigg).
	\end{align*}
	In the sum, the diagonal terms contribute
	\begin{align}
	\widehat\phi(0)\sum_{i\in\J}r(m_i)^2=\widehat\phi(0)\sum_{i\in\J}|\M_i|=\widehat\phi(0)|\M'|\le\widehat\phi(0)|\M|.\label{621}
	\end{align}
	For the off-diagonal terms, we divide the sum according to the values of $|i-j|$ and have
	\begin{align*}
\sum_{\substack{i,j\in\J\\ i\neq j}} r(m_i)r(m_j) \widehat\phi\bigg(\frac{T}{\log T}\log\frac{m_j}{m_i}\bigg)
& = \sum_{l\ge1}\sum_{\substack{i,j\in\J\\ |i-j|=l}} r(m_i)r(m_j) \widehat\phi\bigg(\frac{T}{\log T}\log\frac{m_j}{m_i}\bigg)
\\
& \le \sum_{l\ge1} \sum_{\substack{i,j\in\J\\ |i-j|=l}} r(m_i)r(m_j)
\widehat\phi\bigg(\frac{T}{\log T}\log\Big(1+\frac{\log T}{T}\Big)^{l-1}\bigg).
	\end{align*}
	Recall that $\phi$ is rapidly decay. Thus this is bounded by
\begin{align*}
		\sum_{l\ge1}\sum_{\substack{i,j\in\J\\ |i-j|=l}} r(m_i)r(m_j)\le\sum_{i\in\J}r(m_i)^2\le|\M|.
	\end{align*}
 This combined with (\ref{621}) proves the lemma.
\end{proof}

\begin{lemma}\label{LI}
	For $I_2(R; T)$, we have
\begin{align*}
	I_2(R; T)=M_2(R; T)+O(|\M|T^{\beta+\kappa}\log T),
\end{align*}
{where the implied constant is absolute.}
\end{lemma}
\begin{proof}
By the definition, we have that
\begin{align*}
&I_2(R; T)-M_2(R; T)\\
&\ \ =\left(\int_{|t|<2T^\beta}\int_{\R}
+\int_{2T^\beta\le|t|\le T/2}\int_{|u|>|t|/2}
+\int_{|t|>T/2}\int_{\R}\right)|R(t)|^2{\phi}\bigg(\frac{t\log T}{T}\bigg){\mathfrak Z}_\sigma(t,u)\d u\d t.
\end{align*}
Denote the three integrals on the right-hand side by $D_1(R,T)$, $D_2(R,T)$ and $D_3(R,T)$ respectively. We prove that each is bounded by $O(|\M|T^{\beta+\kappa}\log T)$.

Firstly, recall that
$$
|\zeta(\sigma+\i t)|\ll (1+|t|)^{\frac{1}{3}(1-\sigma)}.
$$
Therefore, we have
\begin{align*}
	\int_{|t|<2T^\beta}\int_{|u|\le T^\beta}|{\mathfrak Z}_\sigma(t,u)|\d u\d t
    \le \int_{|t|<2T^\beta}\int_{|u|\le T^\beta}|\zeta(\sigma+\i t)|^2\d u\d t
    \ll T^\beta\log T,
\end{align*}
and
\begin{align*}
	\int_{|t|<2T^\beta}\int_{|u|>T^\beta}|{\mathfrak Z}_\sigma(t,u)|\d u\d t
    \le \int_{|t|<2T^\beta}\int_{|u|>T^\beta}\frac{(|t|+|u|)^{\frac{1}{3}(1-\sigma)}}{u^2}\d u\d t
    \ll T^\beta.
\end{align*}
Consequently, we derive that
\begin{align*}
	D_1(R,T)\le N^2\int_{|t|<2T^\beta}\left(\int_{|u|\le T^\beta}+\int_{|u|>T^\beta}\right)|{\mathfrak Z}_\sigma(t,u)|\d u\d t\ll|\M|T^{\beta+\kappa}\log T.
\end{align*}
Secondly, we have that $D_2(R,T)$ is bounded by
\begin{align*}
	\int_{T^\beta\le|t|\le T}|R(t)|^2{\phi}\bigg(\frac{t\log T}{T}\bigg)\int_{|u|>|t|/2}\frac{(|t|+|u|)^{\frac{1}{3}(1-\sigma)}}{u^2}\d u\d t\le T^{-\beta/2}I_1(R; T),
\end{align*}
and thus is admissible by Lemma \ref{LI'}.
Finally, since $\phi$ decays rapidly, we have
\begin{align*}
	D_3(R,T)\ll N^2=|\M|T^\kappa.	
\end{align*}
This completes the proof.
\end{proof}

Using Lemma \ref{LI'} and Lemma \ref{LI}, we establish from \eqref{zbetasigma} that
\begin{align}
	\frac{{|M_2(R; T)|}}{M_1(R; T)}\gg\frac{{|I_2(R; T)|}}{T|\M|}\log T+O\left(T^{\beta+\kappa-1}(\log T)^2\right).\label{523}
\end{align}
To estimate $I_2(R; T)$, we need to deal with the convolution of $K(u)$ with $\zeta(s)$.
Here we quote the following lemma due to de la Bret\`ech{e} and Tenenbaum \cite[Lemma 5.3]{DT18}.
\begin{lemma}\label{l522}
	Let $\sigma\in(-\infty,1)$. Suppose $K(z)$ is analytic in the strip $\im z\in [\sigma-2,0]$, satisfying
	$$\sup_{\sigma-2\le y\le0}|K(x+\i y)|\ll \frac{1}{1+x^2}.$$
	Then for any $t\neq0$ we have
	\begin{align*}
		\int_{\R}{\mathfrak Z}_\sigma(t,u)\d u
=\sum_{k,l\ge1}\frac{\widehat K(\log kl)}{k^{\sigma+\i t}l^{\sigma-\i t}}-2\pi\zeta(1-2\i t)K(\i(\sigma+\i t)-\i)-2\pi\zeta(1+2\i t)K(\i(\sigma-\i t)-\i).\end{align*}
\end{lemma}
{We expand $K(u)$ analytic continuously to the whole plane, satisfying the conditions of Lemma \ref{l522}.} Therefore, we may deduce that
$$
I_2(R; T)=I_{2,1}(R; T)-I_{2,2}(R; T)-I_{2,3}(R; T),
$$
where
\begin{align*}
I_{2,1}(R; T)
& := \int_{\R}|R(t)|^2{\phi}\bigg(\frac{t\log T}{T}\bigg)
\sum_{k,l\ge1}\frac{\widehat K(\log kl)}{k^{\sigma+\i t}l^{\sigma-\i t}}\d t,
\\
I_{2,2}(R; T)
& := 2\pi\int_{\R}\zeta(1-2\i t)K(\i(\sigma+\i t)-\i)|R(t)|^2{\phi}\bigg(\frac{t\log T}{T}\bigg)\d t,
\\
I_{2,3}(R; T)
& := 2\pi\int_{\R}\zeta(1+2\i t)K(\i(\sigma-\i t)-\i)|R(t)|^2{\phi}\bigg(\frac{t\log T}{T}\bigg)\d t.
\end{align*}

For $I_{2,2}(R; T)$ and $I_{2,3}(R; T)$, we have
\begin{align}\label{524}
|I_{2,2}(R; T)| + |I_{2,3}(R; T)|
\ll \frac{|\M|T^\kappa}{\log T}\int_{\R}\frac{|\zeta(1\pm2\i t)|}{1+t^2}{\phi}\bigg(\frac{t\log T}{T}\bigg)\d t
\ll \frac{|\M|T^\kappa}{\log T},
\end{align}
since $K(\i(\sigma\pm\i t)-\i)\ll{1}/{((1+t^2)\log T)}$.
Thus $I_{{2,2}}(R,T)$ and $I_{{2,3}}(R,T)$ are admissible as error terms.

For $I_{{2,1}}(R; T)$, each term is nonnegative, so we have
\begin{align*}
I_{{2,1}}(R; T)
& = \frac{T}{\log T}\sum_{m,n\in\M'}r(m)r(n)\sum_{k,l\ge1}\frac{\widehat K(\log kl)}{k^{\sigma}l^{\sigma}}
\widehat{\phi}\bigg(\frac{T}{\log T}\log\frac{mk}{nl}\bigg)
\\
& \ge \frac{T}{\log T}\sum_{m,n\in\M'} r(m)r(n) \sum_{kl\le T^\varepsilon} \frac{\widehat K(\log kl)}{k^{\sigma}l^{\sigma}}\widehat{\phi}\bigg(\frac{T}{\log T}\log\frac{mk}{nl}\bigg).
\end{align*}
Note that by (\ref{Khat}), $kl\le T^\varepsilon$ implies $\widehat K(\log kl)\ge1/2$. Therefore
\begin{align*}
	I_{{2,1}}(R; T)\ge\frac{T}{2\log T}\sum_{kl\le T^\varepsilon}\frac{1}{k^{\sigma}l^{\sigma}}\sum_{m,n\in\M'}r(m)r(n)\widehat{\phi}\bigg(\frac{T}{\log T}\log\frac{mk}{nl}\bigg).
\end{align*}
For the inner sum, for fixed $k$ and $l$,  we have
\begin{align*}
\sum_{m,n\in\M'} r(m)r(n) \widehat{\phi}\bigg(\frac{T}{\log T}\log\frac{mk}{nl}\bigg)
& \ge\sum_{i,j\in\J} \sum_{\substack{m\in \M_i,n\in \M_j\\ mk=nl}}
\widehat{\phi}\bigg(\frac{T}{\log T}\log\frac{m_ik}{m_jl}\bigg)
\\
& = \sum_{i,j\in\J} \sum_{\substack{m\in \M_i,n\in \M_j\\ mk=nl}}
\widehat{\phi}\bigg(\frac{T}{\log T}\log\frac{m_in}{m_jm}\bigg)
\gg \sum_{\substack{m,n\in \M\\ mk=nl}} 1,
\end{align*}
since $1\le m/m_i\le1+\log T/T$ and $1\le n/m_j\le1+\log T/T$.
Thus we have
\begin{align}\label{I1RT}
I_{{2,1}}(R; T)
\gg\frac{T}{\log T}\sum_{\substack{m,n\in \M\\ mk=nl}} \sum_{kl\le T^\varepsilon}\frac{1}{k^{\sigma}l^{\sigma}}.
\end{align}
Write $m=m'\gcd(m,n)$ and $n=n'\gcd(m,n)$.
Then the relation $mk=nl$ implies
$$L=\frac{k}{n'}=\frac{l}{m'},$$
for some integer $L$ and
$$kl=L^2m'n'=L^2\frac{\lcm(m,n)}{\gcd(m,n)}.$$
Therefore in (\ref{I1RT}), we have
\begin{align*}
I_{{2,1}}(R; T)
& \gg\frac{T}{\log T}\sum_{m,n\in \M}\sum_{L^2\frac{\lcm(m,n)}{\gcd(m,n)}\le T^\varepsilon} \frac{1}{L^{2\sigma}}\bigg(\frac{\gcd(m,n)}{\lcm(m,n)}\bigg)^{\sigma}
\\
& \gg \frac{T}{\log T}\sum_{\substack{m,n\in \M\\ \frac{\lcm(m,n)}{\gcd(m,n)}\le T^\varepsilon}}
\bigg(\frac{\gcd(m,n)}{\lcm(m,n)}\bigg)^{\sigma}
\\
& = \frac{T}{\log T}S_\sigma(\M)
- \frac{T}{\log T}\sum_{\substack{m,n\in \M\\ \frac{\lcm(m,n)}{\gcd(m,n)}> T^\varepsilon}}
\bigg(\frac{\gcd(m,n)}{\lcm(m,n)}\bigg)^{\sigma}.
\end{align*}
By Rankin's trick, we have
\begin{align*}
\sum_{\substack{m,n\in \M\\ \frac{\lcm(m,n)}{\gcd(m,n)}> T^\varepsilon}} \bigg(\frac{\gcd(m,n)}{\lcm(m,n)}\bigg)^{\sigma}
\le \frac{1}{ T^{(\sigma-1/2)\varepsilon}} \sum_{m,n\in \M}\bigg(\frac{\gcd(m,n)}{\lcm(m,n)}\bigg)^{1/2}
= \frac{1}{ T^{(\sigma-1/2)\varepsilon}}S_{1/2}(\M).
\end{align*}
By using \eqref{gamma12}, we have that this is bounded by
\begin{align*}
\frac{|\M|}{ T^{(\sigma-1/2)\varepsilon}}\exp\Bigg(\big\{2\sqrt2+o(1)\big\}\sqrt{\frac{\kappa\log T\log_3T}{\log_2T}}\Bigg)
= o(|\M|),
\end{align*}
{since $\sigma\ge1/2+1/\log_2T$ and we can choose $\varepsilon=1/2022$.}
So we have
$$I_{{2,1}}(R; T)\gg\frac{T}{\log T}S_\sigma(\M).$$
Combining this with (\ref{523}) and (\ref{524}), we have
$$\frac{{|M_2(R; T)|}}{M_1(R; T)}\gg\frac{S_\sigma(\M)}{|\M|}+T^{\beta+\kappa-1}(\log T)^2.$$
Choose $\kappa=1-\beta$. By (\ref{zbetasigma}) we have
$$Z_\beta(\sigma,T)^2\ge\frac{{|M_2(R; T)|}}{M_1(R; T)}\gg\frac{S_\sigma(\M)}{|\M|}+(\log T)^2,$$
and thus by taking maximum of both sides over $|\M|=N$ we have
$$
Z_\beta(\sigma,T)\gg\sqrt{\Gamma_\sigma(T^{1-\beta})}.
$$

\vskip 5mm

\noindent{\bf Acknowledgement}.
The authors would like to thank professor Jie Wu, for his suggestion on exploring this subject and his generous help in overcoming some difficulties. They also thank Marc Munsch for pointing to his work \cite{Ai2019}, and Junxian Li and Jing Zhao for some valuable discussions. The authors are grateful to the China Scholarship Council (CSC) for supporting their studies in France. Bin Wei was also funded by
NSFC (Grant No. 11701412).

\end{document}